\documentclass{amsart}
\usepackage{amsmath}
\usepackage{amssymb}
\usepackage{amsthm}
\usepackage{enumerate}
\usepackage[pdftex]{graphicx}
\usepackage{caption}
\theoremstyle{definition}
\newtheorem{definition}{Definition}[section]
\theoremstyle{plain}
\newtheorem{lemma}[definition]{Lemma}
\newtheorem{theorem}[definition]{Theorem}

\newtheorem{proposition}[definition]{Proposition}
\newtheorem{corollary}[definition]{Corollary}
\theoremstyle{remark}
\newtheorem{remark}[definition]{Remark}

\makeatletter
\@namedef{subjclassname@2020}{
  \textup{2020} Mathematics Subject Classification}
\makeatother

\newcommand{\mycl}{\operatorname{cl}}
\newcommand{\myint}{\operatorname{int}}
\newcommand{\myrank}{\operatorname{rank}}
\newcommand{\myIso}{\operatorname{iso}}
\newcommand{\myLpt}{\operatorname{lpt}}
\newcommand{\myedim}{\operatorname{eRank}}
\newcommand{\mylns}{\operatorname{lns}}

\begin{document}
\title[Michael's selection theorem]{Michael's selection theorem in general d-minimal structures}
\author[M. Fujita]{Masato Fujita}
\address{Department of Liberal Arts,
	Japan Coast Guard Academy,
	5-1 Wakaba-cho, Kure, Hiroshima 737-8512, Japan}
\email{fujita.masato.p34@kyoto-u.jp}

\begin{abstract}
	Thamrongthanyalak demonstrated a definable version of Michael's selection theorem  in d-minimal expansions of the real field.
	We generalize this result to the case in which the structures are d-minimal expansions of ordered fields $\mathcal F=(F,<,+,\cdot,0,1,\ldots)$.
	We also show that we can choose a definable continuous selection $f$ of a lower semi-continuous map $T:E \rightrightarrows F$ so that $f(x)$ is contained in the interior of $T(x)$ when the interior is not empty.
\end{abstract}



\maketitle

\section{Introduction}\label{sec:intro}
Thamrongthanyalak proved a definable version of famous Michael's selection theorem \cite{Michael1,Michael2} in d-minimal expansions of the real field \cite{T}.
He posed an open question whether his definable Michael's selection theorem holds in d-minimal expansions of arbitrary ordered fields \cite[3.6]{T}.
We give an affirmative answer to this question in this paper.
In the proof, we heavily use Fornasiero's result on d-minimal expansions of ordered fields \cite{Fornasiero} and  the author's result on extended rank \cite{Fuji_dmin_quotient}.

Let $T$ be a lower semi-continuous set-valued definable map from a definable set $E$ to $F^m$, where $F$ is the universe of the given d-minimal expansion of an ordered field.
Suppose that $T(x)$ is closed and convex for each $x \in E$.
Our definable Michael's selection theorem asserts that we can choose a definable continuous map $f:E \to F^m$ so that $f(x) \in T(x)$ for each $x \in E$.
This is the first main theorem of this paper.
As a next step, we consider the following natural question:
\begin{quotation}
	Can we choose $f$ so that $f(x)$ is contained in the interior of $T(x)$ when $T(x)$ has a nonempty interior in $F^m$?
\end{quotation}
Thamrongthanyalak's proof of definable Michael's selection theorem and our generalization do not guarantee that $f(x)$ is contained in the interior of $T(x)$.
We have succeeded to construct such an $f$ when $m=1$.
This is the second main theorem of this paper.

This paper is organized as follows:
In Section \ref{sec:preliminary}, we recall the definition of d-minimal structures and previous studies on d-minimal structures.
We also prove several basic assertions on d-minimal structures in this section.
Section \ref{sec:michael} is devoted to the proof of definable Michael's selection theorem in an arbitrary d-minimal expansion of an ordered field.
We consider the problem whether we can choose $f$ so that $f(x)$ is contained in the interior of $T(x)$ in Section \ref{sec:improvement}.
The partial result introduced above is proved in this section.
Section \ref{sec:conclusion} gives several concluding remarks.

We introduce the terms and notations used in this paper. 
Throughout, the term ‘definable’ means ‘definable in the given structure with parameters.’
For any expansion of a linear order whose universe is $F$, we assume that $F$ is equipped with the order topology induced from the linear order $<$ and the topology on $F^n$ is the product topology of the order topology on $F$. % unless the topology in consideration is explicitly described.
We put $|x|=\max\{|x_i|\;|\; 1 \leq i \leq n\}$ for each $x=(x_1,\ldots, x_n) \in F^n$.
For a topological space $X$, $\myint_{X}(A)$, $\mycl_{X}(A)$ and $\partial_{X}(A)$ denotes the interior, the closure and the frontier of a subset $A$ of $X$, respectively.
We drop the subscript $X$ when the set $X$ is obvious from the context.

\section{Preliminary}\label{sec:preliminary}
This section is a preliminary section.
We first recall the notion of definable completeness and definable Tietze extension theorem.

\begin{definition}[\cite{M}]
	An expansion of a dense linear order without endpoints $\mathcal F=(F,<,\ldots)$ 
	is \textit{definably complete} if every definable subset of $F$ has both a supremum and an infimum in 
	$F \cup \{ \pm \infty\}$.
\end{definition}

\begin{theorem}[Definable Tietze extension theorem]\label{thm:bounded_tietze}
	Consider a definably complete expansion of an ordered field $\mathcal F=(F,<,+,\cdot,0,1,\ldots)$.
	Let $X$ be a definable subset of $F^n$ and $A$ be a definable closed subset of $X$.
	Then, a definable bounded continuous function $f:A \to F$  has a definable bounded continuous extension $g:X \to F$.
	We can drop the assumption that $f$ is bounded when $X=F^n$.
\end{theorem}
\begin{proof}
	See \cite[Theorem 2.5]{Fuji_dmin_quotient} and \cite[Lemma 6.6]{AF}.
\end{proof}

We next recall the notion of d-minimality.
The notion of d-minimality was first proposed by Miller \cite{Miller-dmin}.
Miller only considered the case in which the universe is the set of reals $\mathbb R$.
The following general definition is given by Fornasiero.

\begin{definition}[\cite{Fornasiero}]
	An expansion of a dense linear order without endpoints $\mathcal F=(F,<,\ldots)$ 
	is \textit{d-minimal} if it is definably complete, and every definable subset $X$ of $F$ is the union of an open set and finitely many discrete sets, where the number of discrete sets does not depend on the parameters of definition of $X$.
\end{definition}

It is already known that several expansions of ordered fields fall into the class of d-minimal structures \cite{D2,FM,MT}.
In addition, every definably complete locally o-minimal structure \cite{TV} is d-minimal.

We collect the assertions in the previous studies.
\begin{lemma}\label{lem:cont_dmin}
	Consider a d-minimal expansion of an ordered field.
	Let $U$ be a definable subset of $F^m$ and $f:U \to F^n$ be a definable map.
	Then the set of points at which $f$ is discontinuous has an empty interior in $F^m$.
	In particular, if $U$ has a nonempty interior in $F^m$, there is a nonempty open box $V$ contained in $U$ such that the restriction of $f$ to $V$ is continuous.
\end{lemma}
\begin{proof}
	See \cite[Theorem 3.10(6), (12)]{Fornasiero}.
\end{proof}

\begin{lemma}\label{lem:definable_selection}
	Consider a d-minimal expansion of an ordered group whose universe is $F$.
	Let $\pi:F^n \to F^m$ be a coordinate projection and $X$ be a definable subset of $F^n$.
	Then there exists a definable map $\varphi:\pi(X) \to X$ such that the composition $\pi \circ \varphi$ is the identity map on $\pi(X)$.
\end{lemma}
\begin{proof}
	See \cite{Miller-choice}.
\end{proof}

We use the notion of extended rank defined in \cite{Fuji_dmin_quotient}.
We first recall the definition of dimension and Cantor-Bendixson rank.
\begin{definition}[Dimension]\label{def:dim}
	Consider an expansion of a dense linear order without endpoints.
	Let $F$ be the universe.
	We consider that $F^0$ is a singleton equipped with the trivial topology.
	Let $X$ be a nonempty definable subset of $F^n$.
	The dimension of $X$ is the maximal nonnegative integer $d$ such that $\pi(X)$ has a nonempty interior for some coordinate projection $\pi:F^n \rightarrow F^d$.
	We set $\dim(X)=-\infty$ when $X$ is an empty set.
\end{definition}

\begin{definition}[Cantor-Bendixson rank\cite{FM}]\label{def:lpt}
	We denote the set of isolated points in a topological space $S$ by $\myIso(S)$.
	We set $\myLpt(S):=S \setminus \myIso(S)$.
	In other word, a point $x \in S$ belongs to $\myLpt(S)$ if and only if $x \in \mycl_S(S \setminus \{x\})$.
	
	Let $X$ be a nonempty closed subset of a topological space $S$.
	We set $X\langle 0 \rangle=X$ and, for any $m>0$, we set $X \langle m \rangle = \myLpt(X \langle m-1 \rangle)$.
	We say that $\myrank(X)=m$ if $X \langle m \rangle=\emptyset$ and $X\langle m-1 \rangle \neq \emptyset$.
	We say that $\myrank X = \infty$ when $X \langle m \rangle \neq \emptyset$ for every natural number $m$.
	We set $\myrank(Y):=\myrank(\mycl_S(Y))$ when $Y$ is a nonempty subset of $S$ which is not necessarily closed.
\end{definition}

%\begin{lemma}\label{lem:dmin_countable}
%	Consider a d-minimal expansion of an ordered field.
%	A nonempty definable set $X$ of dimension zero has an isolated point and its rank is finite.
%\end{lemma}
%\begin{proof}
%	See \cite[Lemma 7.16]{Fuji_dmin_quotient}.
%\end{proof}

We have finished the preparation.
We introduce the extended rank function $\myedim$ for d-minimal structures.

\begin{definition}[Extended rank]\label{def:extended_dim_dmin}
	Consider a d-minimal expansion of an ordered field $\mathcal F=(F,<,+,\cdot,0,1,\ldots)$.
	Let $\Pi(n,d)$ be the set of coordinate projections of $F^n$ onto $F^d$.
	Recall that $F^0$ is a singleton.
	We consider that $\Pi(n,0)$ is a singleton whose element is a trivial map onto $F^0$.
	Since $\Pi(n,d)$ is a finite set, we can define a linear order on it. 
	We denote it by $<_{\Pi(n,d)}$.
	Let $\mathcal E_n$ be the set of triples $(d,\pi,r)$ such that $d$ is a nonnegative integer not larger than $n$, $\pi \in \Pi(n,d)$ and $r$ is a positive integer.
	The linear order $<_{\mathcal E_n}$ on $\mathcal E_n$ is the lexicographic order.
	We abbreviate the subscript $\mathcal E_n$ of $<_{\mathcal E_n}$ in the rest of the paper, but it will not confuse readers.
	Let $X$ be a nonempty bounded definable subset of $F^n$.
	The triple $(d,\pi,r)$ is the \textit{extended rank} of $X$ and denoted by  $\myedim_n(X)$ if it is an element of $\mathcal E_n$ satisfying the following conditions:
	\begin{itemize}
		\item $d = \dim X$;
		\item the projection $\pi$ is a largest element in $\Pi(n,d)$ such that $\pi(X)$ has a nonempty interior;
		\item  $r$ is a largest positive integer such that there exists a definable open subset $U$ of $F^d$ contained in $\pi(X)$ such that the set $\pi^{-1}(x) \cap X$ is of dimension zero and the equality $\myrank(\pi^{-1}(x) \cap X) = r$ holds  for each $x \in U$.
	\end{itemize}
	Note that such a positive integer $r$ exists by \cite[Lemma 5.10]{Fornasiero} and \cite[Theorem 3.10(12)]{Fornasiero}.
	We set $\myedim_n(\emptyset)=-\infty$ and define that $-\infty$ is smaller than any element in $\mathcal E_n$.
	
	Let us consider the case in which $X$ is an unbounded definable subset of $F^n$.
	Let $\varphi:F \to (-1,1)$ be the definable homeomorphism given by $\varphi(x)=\frac{x}{\sqrt{1+x^2}}$.
	We define $\varphi_n:F^n \to (-1,1)^n$ by $\varphi_n(x_1,\dots, x_n)=(\varphi(x_1),\ldots, \varphi(x_n))$.
	We set $\myedim_n(X)=\myedim_n(\varphi_n(X))$.
\end{definition}

\begin{remark}\label{rem:induction}
There is no infinite strictly decreasing sequence $e_1 > e_2 > \cdots$ of elements in $\mathcal E_n$.
Here, an infinite strictly decreasing sequence means a map $e : \mathbb N \to \mathcal E_n$ such that $e(m_1)>e(m_2)$ for each $m_1 < m_2 \in \mathbb N$.
\end{remark}
\begin{proof}
Set $e_1=(d,\pi,r)$.
We show the remark by induction on $d$ and $\pi$.
When $d=0$,  the map $\pi:F^n \to F^0$ is the canonical map.
Elements smaller than $(0,\pi,r)$ is of the form $(0,\pi,s)$ with $s<r$.
We have $r-1$ elements in $\mathcal E_n$ smaller than $e_1$.
This implies that there are no infinite strictly decreasing sequence stating from $e_1$.

Let us consider the case in which  $d>0$.
Set $(d_r,\pi_r, r')=e_r$.
Since we have only $r-1$ elements of the form $(d,\pi,s)$ smaller than $e_1$, we have either $d_r<d$ or $d_r=d \wedge \pi_r<\pi$.
By the induction hypothesis, there are no infinite strictly decreasing sequence starting from $e_r$.
This implies that there are no infinite strictly decreasing sequence starting from $e_1$.
\end{proof}

We collect two basic properties of extended rank proved in \cite{Fuji_dmin_quotient}.

\begin{lemma}\label{lem:frontier_extended}
	Consider a d-minimal expansion of an ordered field $\mathcal F=(F,<,+,\cdot,0,1,\ldots)$.
	We have $\myedim_n \partial X < \myedim_n X$ for each nonempty definable subset $X$ of $F^n$.
\end{lemma}
\begin{proof}
	See \cite[Lemma 7.16]{Fuji_dmin_quotient}.
\end{proof}

\begin{lemma}\label{lem:dmin_union}
	Consider a d-minimal expansion of an ordered field whose universe is $F$.
	Let $A$ and $B$ be definable subsets of $F^n$.
	The equality $\myedim_n(A \cup B)=\max\{\myedim_n(A),\myedim_n(B)\}$ holds.
\end{lemma}
\begin{proof}
	See \cite[Proposition 7.19]{Fuji_dmin_quotient}.
\end{proof}

We also need the following proposition and its corollary:
\begin{proposition}\label{prop:definable_mfd}
	Consider a d-minimal expansion of an ordered field whose universe is $F$.
	Let $X$ be a definable subset of $F^n$ and set $(d,\pi,r)=\myedim_n(X)$.
	There exists a definable open subset $U$ of $X$ satisfying the following:
	\begin{enumerate}
		\item[(1)]  The inequality $\myedim_n(X \setminus U)<\myedim_n(X)$ holds;
		\item[(2)] For any $x \in U$, there exists an open box $B$ containing the point $x$ such that the restriction of $\pi$ to $B \cap X$ is a definable homeomorphism onto $\pi(B)$.
	\end{enumerate}
\end{proposition}
\begin{proof}
	We reduce to simpler cases step by step.
	We may assume that $\pi$ is the projection onto the first $d$ coordinates without loss of generality.
	We define the definable subset $W$ of $\pi(X)$ as the set of points $x \in \pi(X)$ such that the set $\pi^{-1}(x) \cap X$ is of dimension zero and the equality $\myrank(\pi^{-1}(x) \cap X) = r$ holds.
	Let $V$ be the interior of $W$ in $F^d$, which is not empty by the definition of extended rank.
	The definition of extended rank implies that $\myedim_n(X \setminus \pi^{-1}(V))< \myedim_n(X)$.
	Replacing $X$ with $X \cap \pi^{-1}(V)$, by using Lemma \ref{lem:dmin_union}, we may assume that $\pi(X)$ is open and  $\pi^{-1}(x) \cap X$ is of dimension zero and the equality $\myrank(\pi^{-1}(x) \cap X) = r$ holds for each $x \in \pi(X)$.
	
	Set $Y:=\{x \in X\;|\; x \text{ is isolated in }\pi^{-1}(\pi(x)) \cap X\}$.
	It is obvious that $\myrank(\pi^{-1}(x) \cap (X \setminus Y))<\myrank(\pi^{-1}(x) \cap X)$, and this implies that $\myedim_n(X \setminus Y) < \myedim_n(X)$.
	We also have $\myedim_n(\mycl_X(X \setminus Y))<\myedim_n(X)$ by Lemma \ref{lem:frontier_extended} and Lemma \ref{lem:dmin_union}.
	Replacing $X$ with $\myint_X(Y)$ and using Lemma \ref{lem:dmin_union}, we may assume that $\pi^{-1}(x) \cap X$ is discrete for each $x \in \pi(X)$.
	
	Let $U$ be the set of points $x \in X$ such that the restriction of $\pi$ to $B \cap X$ is a definable homeomorphism onto $\pi(B)$ for some open box $B$ containing the point $x$.
	It is obvious that $U$ is open in $X$.
	We have only to prove that $\myedim_n(X \setminus U)<\myedim_n(X)$.
	Assume for contradiction that $\myedim_n(X \setminus U)=\myedim_n(X)=(d,\pi,r)$.
	Set $Z=X \setminus U$ for simplicity of nation.
	Note that $\pi(Z)$ has a nonempty interior by the assumption.
	We can take a definable map $f:\pi(Z) \to Z$ by Lemma \ref{lem:definable_selection}. 
	There exists an open box $B_1$ contained in $\pi(Z)$ such that the map $f$ restricted to  $B_1$ is continuous by Lemma \ref{lem:cont_dmin}.
	Let $g:B_1 \to Z$ be the restriction of $f$ to $B_1$.
	Let $h:B_1\to F$ be the definable function defined by $h(x)=\min\{1,\inf\{|g(x)-u|\;|\; g(x) \neq u \in \pi^{-1}(x) \cap X\}\}$.
	We always have $h(x)>0$ because $\pi^{-1}(x) \cap X$ is discrete by the assumption.
	We may assume that $h$ is continuous by Lemma \ref{lem:cont_dmin} by shrinking $B_1$ if necessary.
	
	Fix an arbitrary point $x \in B_1$.
	Take $\varepsilon>0$ so that $\varepsilon<h(x)$.
	We can choose a small open box $B_2$ containing the point $x$ and contained in $B_1$ such that $h(x')>\varepsilon$ for each $x' \in B_2$ because $h$ is continuous.
	Since $g$ is continuous, we can take an open box $B_3$ containing the point $x$ and contained in $B_2$ such that $|\pi' \circ g(x)- \pi' \circ g(x')|<\varepsilon$ for each $x' \in B_3$, where $\pi':F^n \to F^{n-d}$ is the coordinate projection forgetting the first $d$ coordinates.
	Set $B=B_3 \times \{y \in F^{n-d}\;|\; |y-\pi' \circ g(x)|<\varepsilon\}$.
	It is obvious that the restriction of $\pi$ to $X \cap B=g(B_3)$ is a definable homeomorphism onto $\pi(B)=B_3$.
	This contradicts to the fact that $g(x)$ belongs to $Z$ and the definition of $Z$.
\end{proof}

\begin{corollary}\label{cor:dmin_cont2}
	Consider a d-minimal expansion of an ordered field whose universe is $F$.
	Let $X$ be a definable subset of $F^n$ and $f:X \to F^m$ be a definable map.
	Let $D(f)$ be the set of points at which $f$ is discontinuous.
	Then the inequality $\myedim_n(D(f))<\myedim_n(X)$ holds.
\end{corollary}
\begin{proof}
	Set $(d,\pi,r)=\myedim_n(X)$.
	Let $U$ be a definable open subset of $X$ satisfying the conditions in Proposition \ref{prop:definable_mfd}.
	Set $Z=D(f) \cap U$.
	We want to show that $\pi(Z)$ has an empty interior.
	Assume for contradiction that $\pi(Z)$  has a nonempty interior.
	We can find a definable map $g:\pi(Z) \to Z$ such that $\pi \circ g$ is the identity map on $\pi(Z)$ by Lemma \ref{lem:definable_selection}.
	We can take a nonempty open box $V$ contained in $\pi(Z)$ so that the restriction of $g$ to $V$ is continuous by Lemma \ref{lem:cont_dmin}.
	Take a point $x_0 \in g(V)$ and an open box $B$ containing the point $x_0$ so that the map $\pi$ restricted to $B \cap X$ is a definable homeomorphism onto $\pi(B)$.
	Since $g$ is continuous on $V$, we may assume that $g(V) \subseteq B$ by shrinking $V$ if necessary.
	Set $B'=B \cap \pi^{-1}(V)$.
	We have $X \cap B'=g(V)=Z \cap B'$ because $g(V) \subseteq Z$.
	The restriction $\pi'$ of $\pi$ to $X \cap B'$ is a definable homeomorphism onto $V$ and the restriction $h$ of $g$ to $V$ is its inverse.  
	Consider the composition $f \circ h$.
	We may assume that $f \circ h$ is continuous by Lemma \ref{lem:cont_dmin} by shrinking $V$ if necessary.
	The map $f$ restricted to $X \cap B'=Z \cap B'$ coincides with the composition $f \circ h \circ \pi'$, and it is continuous.
	This contradicts to the definition of $Z$.
	We have demonstrated that $Z=D(f) \cap U$ has an empty interior.
	
	By the definition of extended rank, we have $\myedim_n(D(f) \cap U)<(d,\pi,r)=\myedim_n(X)$.
	We have $\myedim_nD(f) \leq \myedim_n((D(f) \cap U) \cup (X \setminus U))<\myedim_n(X)$ by Lemma \ref{lem:dmin_union}.
\end{proof}

\section{Generalization of definable Michael's selection theorem given by Thamrongthanyalak}\label{sec:michael}

The target of this section is to prove a definable Micheal's selection theorem, which is proved in \cite[Theorem B]{T} when the universe is the set of reals, also holds for any d-minimal expansion of an ordered field.
The author already proved definable Michael's selection theorem in a context different from this case \cite[Theorem 4.10]{FKK} inspired by Thamrongthanyalak's proof.
We employ the same strategy as \cite{FKK}, but we use extended rank instead of dimension.

We first recall basic definitions.

\begin{definition}
	For sets $X$ and $Y$, we denote a map $T$ from $X$ to the power set of $Y$ by $T: X \rightrightarrows Y$ and call it a \textit{set-valued map}.
	When $X$ and $Y$ are topological spaces, a \textit{continuous selection} of a set-valued map $T:X \rightrightarrows Y$ is a continuous function $f:X \rightarrow Y$ such that $f(x) \in T(x)$ for all $x \in X$.  
	
	Consider an expansion $\mathcal F=(F,<,\ldots)$ of a dense linear order without endpoints.
	Let $E$ be a definable subset of $F^n$.
	We consider a \textit{definable set-valued map} $T:E \rightrightarrows F^m$; that is, the \textit{graph} $\Gamma(T):=\bigcup_{x \in E}\{x\} \times T(x) \subseteq E \times F^m$ is definable.
	We define a \textit{definable continuous selection} of $T$, similarly.
	
	A set valued map $T:E \rightrightarrows F^m$ is \textit{lower semi-continuous} if, for any $x_0 \in E$, $y_0 \in T(x_0)$ and a neighborhood $V$ of $y_0$, there exists a neighborhood $U$ of $x_0$ such that $T(x) \cap V \neq \emptyset$ for all $x \in U$.
	A lower semi-continuous set-valued map $T:E \rightrightarrows F^m$ is \textit{continuous} if its graph $\Gamma(T)$ is closed in $E \times F^m$.
\end{definition}

We prepare several definitions and lemmas for proving Theorem \ref{thm:michael}.
\begin{definition}
	Consider a definably complete expansion of an ordered field whose universe is $F$.
	The notation $|\!| x|\!|$ denotes the Euclidean norm in $F^m$.
	Let $E$ be a subset of $F^n$ and $T:E \rightrightarrows F^m$ be a set-valued map.
	Let $f:E \rightarrow F^m$ be a map.
	The notation $T-f:E \rightrightarrows F^m$ denotes the set-valued map given by $(T-f)(x)=\{y-f(x)\;|\; y \in T(x)\}$.
	
	We define the \textit{least norm selection of $T$} when $T(x)$ are closed and convex for all $x \in E$.
	The unique point $y \in T(x)$ whose Euclidean norm $|\!| y|\!|$ is the smallest in $T(x)$ is denoted by $\mylns_T(x)$.
	The existence and uniqueness easily follow from the assumption that $T(x)$ are closed and convex.
	It induces a map $\mylns_T:E \rightarrow F^m$.
\end{definition}

The following lemma is proven straightforwardly.
\begin{lemma}\label{lem:michael_b}
	Let $T: E \rightrightarrows F^m$ be a lower semi-continuous set-valued map and $f:E \rightarrow F^m$ be a continuous map.
	Then, the set-valued map $T-f$ is also lower semi-continuous.
	Furthermore, $T-f$ is continuous when $T$ is continuous.
\end{lemma}

We also need the following lemma:
The proof given here is almost the same as that of \cite[Lemma 4.13]{FKK}, but the author gives a complete proof here for readers' convenience.
\begin{lemma}\label{lem:michael_lns}
	Let $T: E \rightrightarrows F^m$ be a definable continuous set-valued map such that $T(x)$ are closed and convex for all $x \in E$.
	Then, the least norm selection $\mylns_T:E \rightarrow F^m$ is a definable continuous selection of $T$.
\end{lemma}
\begin{proof}
	Definability is obvious.
	We have $\mylns_T(x) \in T(x)$ because $T(x)$ is closed and convex.
	The remaining task is to demonstrate that $\mylns_T$ is continuous. 
	
	Assume for contradiction that $\mylns_T$ is discontinuous at $a \in E$.
	Let $G$ be the graph of $\mylns_T$ and $G'$ be its closure in $E \times F^m$.
	Since $G \subseteq \Gamma(T)$ and $\Gamma(T)$ is closed in $E \times F^m$, the set $G'$ is also contained in the graph $\Gamma(T)$.
	We can take $b \in M^n$ such that $(a,b) \in G' \setminus G$ because $\mylns_T$ is discontinuous at the point $a \in E$.
	The point $b$ is contained in $T(a)$.
	
	Set $\varepsilon =(|\!| b |\!|-|\!| \mylns_T(a) |\!|)/2$.
	The definition of $\mylns_T$ yields that $\varepsilon >0$.
	Due to the lower semi-continuity of $T$, there exists a definable open neighborhood $U$ of $a$ such that, for any $x \in U$, we have $|\!| y_x - \mylns_T(a) |\!|<\varepsilon$ for some $y_x \in T(x)$.
	In particular, we get $|\!|\mylns_T(x)|\!| \leq |\!|y_x |\!| < |\!|\mylns_T(a)|\!|+\varepsilon = (|\!| b |\!|+|\!| \mylns_T(a) |\!|)/2$ for all $x \in U$.
	On the other hand, the point $(a,b)$ is in the closure of the graph $G$ of $\mylns_T$.
	We can take $x_0 \in U$ so that $ |\!| \mylns_T(x_0) -b  |\!| < \varepsilon$.
	It implies that $ |\!| \mylns_T(x_0) |\!| \geq  |\!| b  |\!| -  |\!| \mylns_T(x_0) -b  |\!| >  |\!| b  |\!| -\varepsilon = (|\!| b |\!|+|\!| \mylns_T(a) |\!|)/2$.
	We have obtained a contradiction.
\end{proof}

We are now ready to prove a definable Michael's selection theorem.
\begin{theorem}[Definable Michael's selection theorem]\label{thm:michael}
	Consider a d-minimal expansion $\mathcal F=(F,<,+,\cdot, 0,1,\ldots)$ of an ordered field.
	Let $E$ be a definable closed subset of $F^n$ and $T:E \rightrightarrows F^m$ be a definable lower semi-continuous set-valued map such that $T(x)$ are closed and convex for all $x \in E$.
	The set-valued map $T$ has a definable continuous selection.
\end{theorem}
\begin{proof}
	Set $(d,\pi,r)=\myedim_n(E)$.
	Let $C$ be a definable open subset of $E$ satisfying the following conditions:
	\begin{enumerate}
		\item[(1)]  The inequality $\myedim_n(E \setminus C)<\myedim_n(E)$ holds;
		\item[(2)] For any $x \in C$, there exists an open box $B$ containing the point $x$ such that the restriction of $\pi$ to $B \cap E$ is a definable homeomorphism onto $\pi(B)$.
	\end{enumerate}
	It exists by Proposition \ref{prop:definable_mfd}.
	Note that $\myedim_n(C)=(d,\pi,r)$ by condition (1) and Lemma \ref{lem:dmin_union}.
	The notation $\Gamma(T|_C)$ denotes the graph of the restriction of $T$ to $C$.
	The notation $\partial \Gamma(T|_C)$ denotes its frontier.
	Let $\Pi:F^{n+m} \rightarrow F^n$ be the projection onto the first $n$ coordinates.
	Set $S=\Pi(\partial \Gamma(T|_C))$.
	We put $d_{T(x)}(y)=\inf \{|y-y'|\;|\; y'\in T(x)\}$ for each $y \in F^m$.
	We first demonstrate the following claim:
	\medskip
	
	\textbf{Claim 1.} $\myedim_n S < \myedim_n C$.
	
	\begin{proof}[Proof of Claim 1]
	Assume for contradiction that $\myedim_n S \geq \myedim_n C$.
	Since $S$ is contained in the closure of $C$, we have $\myedim_n S \leq \myedim_n C$ by Lemma \ref{lem:frontier_extended} and Lemma \ref{lem:dmin_union}.
	We have $\myedim_n S = \myedim_n C=(d,\pi,r)$.
	%Set $(d',\pi',r')=\myedim_n C$.
	We also have $\myedim_n (S \cap C)=(d,\pi,r)$ because $\myedim_n \partial C < \myedim_n C$ by Lemma \ref{lem:frontier_extended} and $S$ is contained in the closure of $C$.
	Take a definable map $f:S \rightarrow F^{m}$ such that $(x,f(x)) \in \partial \Gamma(T|_C)$ for all $x \in S$ by Lemma \ref{lem:definable_selection}.
	Let $g:S \to \partial \Gamma(T|_C)$ be the definable map given by $g(x)=(x,f(x))$.
	We have $d_{T(x)}(f(x))>0$ because $T(x)$ is closed.
	The definable subset $D$ of $S \cap C$ of points at which at least one of $f$ and the map given by $x \mapsto d_{T(x)}(f(x))$ is not continuous is of extended rank smaller than $(d,\pi,r)$ by Corollary \ref{cor:dmin_cont2} and Lemma \ref{lem:dmin_union}.
	We have $\myedim_n(\mycl_X(D)) < (d,\pi,r)$ by Lemma \ref{lem:frontier_extended} and Lemma \ref{lem:dmin_union}.
	We have $\myedim_n ((S \cap C) \setminus \mycl_X(D)) = (d,\pi,r)$ by Lemma \ref{lem:dmin_union}.
	By Proposition \ref{prop:definable_mfd}, we can choose a point $z \in (S \cap C) \setminus \mycl_X(D)$ and an open box $B_1$ containing the point $z$ so that the restriction of $\pi$ to $B_1 \cap (S \cap C \setminus \mycl_X(D))$ is a definable homeomorphism onto $\pi(B_1)$.
	We may assume that $B_1 \cap \mycl_X(D)=\emptyset$ by shrinking $B_1$ if necessary.
	%We may assume more.
	We can take an open box $B_2$ containing the point $z$ so that the restriction of $\pi$ to $B_2 \cap C$ is a definable homeomorphism onto $\pi(B_2)$ by condition (2).
	The image $\pi(B_2 \cap (B_1 \cap S \cap C))$ is open and contains the point $\pi(z)$ because the restriction of $\pi$ to $B_1 \cap S \cap C$ is a definable homeomorphism onto $\pi(B_1)$.
	Choose an open box $B_3$ containing $\pi(z)$ and contained in $\pi(B_2 \cap (B_1 \cap S \cap C))$.
	Set $B'=B_1 \cap B_2 \cap \pi^{-1}(B_3)$.
	We have $z \in B'$ and $B' \cap C=B'  \cap S \cap C$.
	Furthermore, the restriction of $\pi$ to $B' \cap C$ is a definable homeomorphism onto $\pi(B')$.
	
	Choose a nonempty bounded closed box $Z$ which is a neighborhood of the point $z$ and contained in $B'$.
	Set $G=g(Z)$.
	Since $g$ is continuous, the definable set $G$ is closed and bounded in $F^{n+m}$ by \cite[Proposition 1.10]{M}.
	We finally constructed the definable subset $G$ of $S \cap C$ such that
	\begin{enumerate}
		\item[(a)] $G$ is closed and bounded in $F^{n+m}$,
		\item[(b)] the interior of $G$ in $C$ is not empty, and 
		\item[(c)] the restrictions of $f$ and the map given by $x \mapsto d_{T(x)}(f(x))$ to $G$ are continuous.
	\end{enumerate}
	We can take a positive $u \in F$ so that $d_{T(x)}(f(x))>u$ for all $x \in G$ by \cite[Corollary]{M}.
	Take a point $x_0$ in the interior of $G$ in $C$ and set $\varepsilon=u/3$.
	There exists an open neighborhood $U'$ of $x_0$ in $S$ contained in $G$ such that $|f(x)-f(x_0)|<\varepsilon$ for all $x \in U'$ because the restriction of $f$ to $G$ is continuous.
	By condition (b), the set $U'$ is also an open neighborhood of $x_0$ in $C$, shrinking $U'$ if necessary.
	We can take $(x,y) \in \Gamma(T|_C)$ such that $x \in U'$ and $|y-f(x_0)|<\varepsilon$ because $(x_0,f(x_0))$ is in the closure of $\Gamma(T|_C)$.
	We get $d_{T(x)}(f(x)) \leq |y-f(x)| \leq |y-f(x_0)|+|f(x)-f(x_0)|< 2\varepsilon = 2u/3$.
	We have obtained a desired contradiction.
	\end{proof}
	
	\textbf{Claim 2.} There exists a partition $E=C_1 \cup \ldots \cup C_k$ of $E$ into definable sets such that $C_j$ is open in $\bigcup_{i=j}^kC_i$ and the restriction $T|_{C_j}:C_j \rightrightarrows F^m$ to $C_j$ is continuous for any $1 \leq j \leq k$.
	
	\begin{proof}[Proof of Claim 2]
		We prove it by induction on $\myedim_n E$.
		It is possible thanks to Remark \ref{rem:induction}.
		When $E$ is of extend rank $(0,\pi_0,1)$, $E$ is discrete and closed by the assumption.
		The set-valued map $T$ is continuous in this case.
		We consider the case in which $\myedim_n E>(0,\pi_0,1)$.
		Let $C$ be a definable open subset of $E$ satisfying conditions (1) and (2).
		Set $S=\Pi(\partial \Gamma(T|_C))$.
		Set $C_1=C \setminus \mycl_E(S)$.
		It is obvious that the restriction of $T$ to $C_1$ is continuous.
		By Claim 1, condition (1), Lemma \ref{lem:frontier_extended} and Lemma \ref{lem:dmin_union}, the definable set $E'= (E \setminus C) \cup \mycl_E(S)$ is of extend rank smaller than $\myedim_n E$.
		Apply the induction hypothesis to $E'$.
		We get the desired partition into definable sets.
	\end{proof}
	
	We finally prove the theorem by induction on the number of definable sets $k$ given in Claim 2.
	The theorem immediately follows from Lemma \ref{lem:michael_lns} when $k=1$.
	We consider the case in which $k>1$.
	Set $D=\bigcup_{i=2}^{k}C_i$, which is closed in $E$.
	The set $D$ is also closed in $F^n$ because $E$ is closed in $F^n$.
	We can take a definable continuous selection $f_1:D \rightarrow F^m$ of the restriction $T|_D$ of $T$ to $D$ by the induction hypothesis.
	Let $f_2:F^n \rightarrow F^m$ be a definable continuous extension of $f_1$ given by Theorem \ref{thm:bounded_tietze}.
	The map $f_3$ is its restriction to $E$.
	We may assume that $f_3$ is constantly zero considering $T-f_3$ in place of $T$ by Lemma \ref{lem:michael_b}.
	
	The least norm selection $\mylns_T$ is definable by Lemma \ref{lem:michael_lns}.
	The least norm selection $\mylns_{T|_{C_1}}$ for the restriction $T|_{C_1}$ coincides with the restriction of $\mylns_T$ to $C_1$ by the definition of least norm selections.
	It is continuous by Claim 2 and Lemma \ref{lem:michael_lns}.
	Consider the definable map $f:E \rightarrow F^m$ given by
	$$
	f(x)=\left\{
	\begin{array}{ll} \mylns_T(x) & \text{ when }x \in C_1,\\ 0 & \text{ elsewhere.}\end{array}
	\right.
	$$
	We have only to show that $f$ is continuous.
	The map $f$ is obviously continuous off the frontier $\partial C_1 \cap E$ of $C_1$ in $E$.
	Fix an arbitrary point $x_0 \in \partial C_1 \cap E$.
	We have $f(x_0)=0$.
	Since $T$ is lower semi-continuous, for any $\varepsilon>0$, there exists $\delta>0$ such that we can take $y_x \in T(x)$ with $|\!|y_x|\!|<\varepsilon$ for all $x \in E$ with $|\!|x-x_0|\!|<\delta$.
	We have $|\!| \mylns_T(x) |\!| \leq |\!|y_x|\!|<\varepsilon$ by the definition of least norm selections.
	It implies that $f$ is continuous.
\end{proof}

\section{Improvement of definable Michael's selection theorem -- partial result}\label{sec:improvement}

We next consider the problem whether we can take a definable continuous selection $f$ so that $f(x) \in \myint(T(x))$ whenever $T(x)$ has a nonempty interior.

Recall that a function $f:X \to F \cup \{\pm \infty\}$ is \textit{lower (resp. upper) semi-continuous} if the set $\{x \in X\;|\; f(x)>a\}$ (resp. $\{x \in X\;|\; f(x)<a\}$) is open for each $a \in F \cup  \{\pm \infty\}$.

We first recall trivial facts.
\begin{lemma}\label{lem:well_known_semicont}
Consider an expansion of dense linear order without endpoints.
The sum of two lower (resp. upper) semi-continuous functions is again lower (resp. upper) semi-continuous.
\end{lemma}
\begin{proof}
	It is a well-known fact.
\end{proof}

\begin{lemma}\label{lem:zeroset}
	Consider a definably complete expansion of an ordered group whose universe is $F$.
	Let $X$ be a definable set and $A$ be a definable closed subset of $X$.
	There exists a nonnegative definable continuous function $f:X \to F$ whose zero set is $A$.
\end{lemma}
\begin{proof}
	Define $f:X \to F$ by $f(x)=\inf\{{|y-x|\;|\; y \in A}\}$.
	It is a routine to check that $f^{-1}(0)=A$ and $f$ is continuous. 
	We omit the details.
\end{proof}

\begin{lemma}\label{lem:zeroset2}
	Consider a definably complete expansion of an ordered field whose universe is $F$.
	Let $X$ be a definable set and $A_1$ and $A_2$ be definable closed subsets of $X$ with $A_1 \cap A_2 = \emptyset$.
	There exists a definable continuous function $f:X \to [0,1]$ such that $f^{-1}(0)=A_1$ and $f^{-1}(1)=A_2$.
\end{lemma}
\begin{proof}
	Let $f_i:X \to F$ be a nonnegative definable continuous function whose zero set is $A_i$ for each $i=1,2$, which exists by Lemma \ref{lem:zeroset}.
	We define $f:X \to [0,1]$ by $f(x)=f_1(x)/(f_1(x)+f_2(x))$.
	This function satisfies the requirement.
\end{proof}

\begin{proposition}\label{prop:obvious}
	Consider a definably complete structure whose universe is $F$.
	Let $T:E \rightrightarrows F$ be a lower semi-continuous set-valued definable map.
	Consider the definable functions  $f,g:E \to F \cup \{\pm \infty\}$ defined by $f(x)=\inf T(x)$ and $g(x)=\sup T(x)$.
	Then, $f$ is upper semi-continuous and $g$ is lower semi-continuous.
\end{proposition}
\begin{proof}
	We only prove that $g$ is lower semi-continuous.
	We can prove that $f$ is upper semi-continuous in the same manner.
	We have only to prove that $G_a:=\{x \in E\;|\; g(x)>a\}$ is open for each $a \in F$.
	Fix an arbitrary $a \in F$ such that $G_a$ is not empty.
	Take an arbitrary point $x_0 \in G_a$.
	There exists an open subset $U$ of $x_0$ such that $T(x) \cap \{y \in F\;|\; y>a\}$ is not empty for each $x \in U$.
	We have $g(x)>a$ for each $x \in U$ by the definition of $g$.
	We have demonstrated that $U \subseteq G_a$, and this means that $G_a$ is open. 
\end{proof}

Consider a lower semi-continuous definable set-valued map $T: E \rightrightarrows F$ whose image $T(x)$ is closed and convex for each $x \in E$. 
A definable continuous selection $h:E \to F$ of $T: E \rightrightarrows F$ satisfies the additional condition $h(x) \in \myint T(x)$ whenever $\myint T(x) \neq \emptyset$ if and only if  $f \leq h \leq g$ on $E$, and $f(x)<h(x)<g(x)$ whenever $f(x)<g(x)$, where $f$ and $g$ are definable maps defined in Proposition \ref{prop:obvious}.
The following theorem together with Proposition \ref{prop:obvious} asserts that we can find such a definable continuous selection $h$.
A similar result in semialgebraic context is found in \cite[Lemma 3.1]{CPS}.

\begin{theorem}\label{thm:weak_Michael}
	Consider a d-minimal expansion of an ordered field $\mathcal F=(F,<,+,\cdot,0,1,\ldots)$.
	Let $E$ be a definable set and $f,g:E \to F \cup \{\pm \infty\}$ be definable maps such that the interval $[f(x),g(x)]$ is not empty for each $x \in X$.
	Assume further that $f$ is upper semi-continuous and $g$ is lower semi-continuous.
	Then there exists a definable continuous function $h:E \to F$ such that $f \leq h \leq g$ on $E$.
	Moreover, we can find $h$ so that $f(x)<h(x)<g(x)$ whenever $f(x)<g(x)$.
\end{theorem}
\begin{proof}
	We reduce to simpler cases step by step.
	We first reduce to the case in which $f$ and $g$ are bounded.
	Let $\tau:F \to (-1,1)$ be the definable homeomorphism defined by $\tau(x)=\frac{x}{\sqrt{1+x^2}}$.
	We define the definable map $\widetilde{f}:E \to [-1,1]$ by 
	\begin{align*}
		\widetilde{f}(x) &= \left\{
			\begin{array}{cl}
				\tau(f(x)) & \text{ when } x \in F,\\
				1 & \text{ when } x= \infty,\\
				-1 & \text{ when } x= -\infty.
			\end{array}
		\right.
	\end{align*}
	Observe that the function $\widetilde{f}$ is upper semi-continuous.
	We define $\widetilde{g}$ in the same manner.
	We may assume that both $f$ and $g$ are bounded considering $\widetilde{f}$ and $\widetilde{g}$ instead of $f$ and $g$, respectively.
	
	Set $(d,\pi,r)=\myedim_n(E)$.
	We demonstrate the theorem by induction on $(d,\pi,r)$.
	When $d=0$ and $r=1$, $E$ is discrete.
	Both $f$ and $g$ are continuous in this case.
	We define $h:E \to F$ by $h(x)=(f(x)+g(x))/2$ in this case.
	It is obvious that $h$ is a definable continuous function satisfying the requirements.
	
	We consider the other case.
	There exists a definable open subset $U$ of $E$ such that $\myedim_n(E \setminus U)<\myedim_n E$ and both $f$ and $g$ are continuous on $U$ by Corollary \ref{cor:dmin_cont2}.
	By the induction hypothesis, there exists a definable continuous map $h_0:E \setminus U \to F$ such that $f \leq h_0 \leq g$ on $E \setminus U$ and $f(x)<h_0(x)<g(x)$ whenever $x \in E \setminus U$ and $f(x)<g(x)$.
	Note that $h_0$ is a bounded function because $f$ and $g$ are bounded.
	Apply Theorem \ref{thm:bounded_tietze} to $h_0$.
	There exists a definable bounded continuous extension $h_1: E \to F$ of $h_0$.
	We may assume that $h_1$ is identically zero on $E$ by considering $f-h_1$ and $g-h_1$ in place of $f$ and $g$, respectively, by Lemma \ref{lem:well_known_semicont}.
	We have reduced to the case in which 
	\begin{enumerate}
		\item[(i)] $U$ is a definable open subset of $E$ and $f$ and $g$ are continuous on $U$;
		\item[(ii)] $f \leq 0 \leq g$ on $E \setminus U$;
		\item[(iii)] $f(x)<0<g(x)$ whenever $x \in E \setminus U$ and $f(x)<g(x)$.
	\end{enumerate}
		
	Consider the definable subset $X:=\{x \in U\;|\; f(x)=g(x)\}$.
	The set $X$ is a closed subset of $U$ because $f$ and $g$ are continuous on $U$.
	Consider the definable function $h_2: (E \setminus U) \cup X \to F$ defined by 
	\begin{align*}
		h_2(x) &= \left\{\begin{array}{ll}
			0 & \text { when } x \in E \setminus U,\\
			f(x)(=g(x)) & \text{ on }X.
		\end{array}
		\right.
	\end{align*}
	The function $h_2$ is obviously continuous off $(E \setminus U) \cap \partial_E(X)$.
	Take an arbitrary point $x_0 \in (E \setminus U) \cap \partial_E(X)$.
	We have $f(x_0) \leq 0 \leq g(x_0)$ by condition (ii) because $x_0 \in E \setminus U$.
	Assume for contradiction that $f(x_0)<0$.
	We have $g(x_0)>0$ by condition (iii).
	By semi-continuity of $f$ and $g$, we have $f(x)<0<g(x)$ for arbitrary $x \in E$ sufficiently close to $x_0$.
	This contradicts the assumption that $x_0 \in (E \setminus U) \cap \partial_E(X)$.
	We have demonstrated that $f(x_0)=g(x_0)=0$.
	Take an arbitrary $\varepsilon >0$.
	There exists a definable neighborhood $N_1$ of $x_0$ in $X$ such that $f(x)<\varepsilon$ on $N_1$ because $f$ is upper semi-continuous.
	We can take a definable neighborhood $N_2$ of $x_0$ in $X$ such that $g(x)>-\varepsilon$ on $N_2$ because $g$ is lower semi-continuous.
	It is easy to check that $|h_2(x)|<\varepsilon$ on $(N_1 \cap N_2) \cup (E \setminus U)$, which is a neighborhood of $x_0$ in $ (E \setminus U) \cup X$.
	This implies that $h_2$ is continuous at $x_0$.
	Consequently, the definable function $h_2$ is continuous.
	Note that $(E \setminus U) \cup X$ is closed in $E$ because $X$ is closed in $U$ and $\partial_E(X)$ is contained in $E \setminus U$.
	Apply Theorem \ref{thm:bounded_tietze} to $h_2$.
	There exists a definable bounded continuous extension $h_3: E \to F$ of $h_2$.
	We may assume that $h_3$ is identically zero on $E$ by considering $f-h_3$ and $g-h_3$ in place of $f$ and $g$, respectively, by Lemma \ref{lem:well_known_semicont}.
	Set $V=U \setminus X$.
	We have finally reduced to the following case:
	\begin{enumerate}
		\item[(a)] $V$ is definable and open in $E$;
		\item[(b)]  $f$ and $g$ are continuous on $V$ and $f(x)<g(x)$ on $V$;
		\item[(c)] We have $f \leq 0 \leq g$ on $E \setminus V$ and $f(x)<0<g(x)$ whenever $x \in E \setminus V$ and $f(x)<g(x)$.
	\end{enumerate}
	
	Put $Z_1=\{x \in E\;|\; f(x) \geq 0\}$ and $Z_2=\{x \in E\;|\;g(x) \leq 0\}$.
	Note that $Z_1$ and $Z_2$ are closed in $E$ by semi-continuity of $f$ and $g$.
	We set $$Y=\{x \in E\;|\; f(x)<0<g(x)\}.$$
	The definable set $Y$ is open in $E$ because $f$ and $g$ are upper and lower semi-continuous, respectively.
	We have $Z_1 \cap Z_2 =\{x \in E\;|\; f(x)=g(x)=0\}$ because $f \leq g$ on $E$.
	We get $Z_1 \cap Z_2 \subseteq E \setminus V$ by condition (b).
	Therefore, we obtain a partition 
	\begin{equation}
	V= (Y \cap V) \cup (Z_1 \cap V) \cup (Z_2 \cap V) \label{eq:partiton}
	\end{equation}
	 of $V$.
	The definable sets $\partial_E(V) \cap Z_i$ are closed in $E$ for $i=1,2$.
	We can find nonnegative definable continuous functions $\eta_i:E \to F$ whose zero sets are $\partial_E(V) \cap Z_i$ by Lemma \ref{lem:zeroset}.	
	Set $$S=Z_1 \cup Z_2 \cup (E \setminus V).$$
	Let us consider the definable map $h_4: S \to F$ defined by
	\begin{align*}
		h_4(x) &= \left\{\begin{array}{cl}
			\min \{f(x)+\eta_1(x),(f(x)+g(x))/2\} & \text{ when } x \in Z_1 \cap V,\\
			\max \{g(x)-\eta_2(x),(f(x)+g(x))/2\} & \text{ when } x \in Z_2 \cap V,\\
			0 & \text{ when } x \in E \setminus V.
		\end{array}\right.
	\end{align*}
	Observe that 
	\begin{equation}
	f(x)<h_4(x)<g(x) \label{eq:ineq_h4}
	\end{equation}
	whenever $x \in V \cap (Z_1 \cup Z_2)$ because $f<g$ on $V$ by condition (b). 
	We have $h_4> 0$ on $Z_1 \cap V$ and $h_4<0$ on $Z_2 \cap V$ by inequality (\ref{eq:ineq_h4}).
	We want to show that $h_4$ is continuous.
	It is obvious that $h_4$ is continuous off $\partial_E(V) \cap (Z_1 \cup Z_2)$.
	Pick an arbitrary point $x_0 \in \partial_E(V) \cap (Z_1 \cup Z_2)$.
	We have only to prove that $h_4$ is continuous at $x_0$.
	Fix an arbitrary $\varepsilon>0$.
	When $x_0 \in Z_1$, we have $f(x_0) = 0$ by condition (c) because $x_0 \notin V$.
	We can take a definable open neighborhood $N_3$ of $x_0$ in $E$ such that $f<\varepsilon/2$ and $\eta_1< \varepsilon/2$ on $N_3$ because $f$ is upper semi-continuous and $\eta_1$ is continuous. 
	We set $N_3=X \setminus Z_1$ when $x_0 \notin Z_1$.
	We have constructed a definable open neighborhood $N_3$ of $x_0$ in $E$ such that $0 \leq h_4(x) = \min \{f(x)+\eta_1(x),(f(x)+g(x))/2\} \leq f(x)+\eta_1(x) <\varepsilon$ whenever $x \in N_3 \cap (Z_1 \cap V)$.
	We can construct a definable open neighborhood $N_4$ of $x_0$ in $E$ such that $0 \leq - h_4(x) <\varepsilon$ whenever $x \in N_4 \cap (Z_2 \cap V)$ in the same manner.
	It is obvious that $h_4(x)=0$ whenever $x \in N_3 \cap N_4 \cap (E \setminus V)$. 
	These inequalities imply that $h_4$ is continuous at $x_0$; and consequently, the definable function $h_4$ is continuous everywhere.
	
	Note that $S$ is closed in $E$.
	Observe that the partition of $V$ given in (\ref{eq:partiton}) implies the equality $Y \cap V=E \setminus S$.
	The definable function $h_4$ is bounded by inequality (\ref{eq:ineq_h4}) because $f$ and $g$ are bounded.
	Apply Theorem \ref{thm:bounded_tietze} to $h_4$.
	There exists a definable continuous extension $h_5$ of $h_4$ to $E$.  
	Set $$W=\{x \in V\;|\; h_5(x) \leq f(x) \text{ or } h_5(x) \geq g(x)\}.$$
	The inequalities $f(x)<h_5(x)<g(x)$ hold for every $x \in V \cap (Z_1 \cup Z_2)$ by inequality (\ref{eq:ineq_h4}).
	The definable set $W$ has an empty intersection with $V \cap (Z_1 \cup Z_2)$.
	There exists a definable continuous function $\delta:V \to [0,1]$ such that $\delta^{-1}(0)=W$ and $\delta^{-1}(1)=V \cap (Z_1 \cup Z_2)$ by Lemma \ref{lem:zeroset2}. 
	We define the definable function $h:E \to F$ by
	\begin{align*}
		h(x) &= \left\{ \begin{array}{cl} h_5(x) & \text{ when } x\in S;\\
		\delta(x)h_5(x) & \text{ when }x \in V.
		\end{array}\right.
	\end{align*} 
	It is a well-defined function because $\delta$ is identically one on $S \cap V=  V \cap (Z_1 \cup Z_2)$. 
	We first prove that $f \leq h \leq g$ on $E$ and $f(x)<h(x)<g(x)$ whenever $f(x)<g(x)$.
	These inequalities follow from condition (c) and inequality (\ref{eq:ineq_h4}) when $x \in S$.
	When $x \in Y \cap V \setminus W$, we have $f(x)<0<g(x)$ and $f(x) <h_5(x)<g(x)$.
	It is obvious that $f(x) <h(x)=\delta(x)h_5(x)<g(x)$ because $0 \leq \delta(x) \leq 1$.
	When $x \in Y \cap V \cap W$, we have $f(x)<0=h(x)<g(x)$ by the definition of $Y$.

	The remaining task is to show that $h$ is continuous.
	It is obvious that $h$ is continuous off $\partial_E(V) $.
	We prove that $h$ is continuous at each point $x_0 \in \partial_E(V) $.
	We have $h(x_0)=h_5(x_0)=h_4(x_0)=0$ because $x_0 \notin V$.
	Take a sufficiently small $\varepsilon>0$.
	We have $|h(x)| \leq |h_5(x)|<\varepsilon$ when $x$ is sufficiently close to $x_0$ because $h_5$ is continuous at $x_0$.
	This implies that $h$ is continuous at the point $x_0$.
\end{proof}

\section{Concluding remarks}\label{sec:conclusion}

We give several comments in this section.
\medskip

\textbf{\ref{sec:conclusion}.1.} 
A definable Michael's selection theorem is unavailable when the multiplication is not definable \cite{AT}.
\medskip

\textbf{\ref{sec:conclusion}.2.} 
We assume that $E$ is closed in Theorem \ref{thm:michael}.
We need this assumption so as to apply Theorem \ref{thm:bounded_tietze} for the case in which $X=F^n$.
We do not need this assumption when $T$ is bounded; that is, there exists $R>0$ such that $T(x) \subseteq (-R,R)^m$ for each $x \in E$.
Theorem \ref{thm:michael} also holds without the assumption that $E$ is closed when $\mathcal F$ is o-minimal.
We can use \cite[Chapter 8, Corollary 3.10]{D} instead of Theorem \ref{thm:bounded_tietze}.
Theorem \ref{thm:weak_Michael} does not assume that $E$ is closed, neither.
These examples illustrate that we can drop the assumption that $E$ is closed in several cases.
The following question arises naturally:
\begin{quote}
Does Theorem \ref{thm:michael} hold even when we drop the assumption that $E$ is closed?
\end{quote}
\medskip

\textbf{\ref{sec:conclusion}.3.} 
We gave a partial result on the following question in Theorem \ref{thm:weak_Michael}, but it is not solved completely:
\begin{quote}
	Can we choose a definable selection $f$ so that $f(x)$ is contained in the interior of $T(x)$ when $T(x)$ has a nonempty interior in Theorem \ref{thm:michael}?
\end{quote}

\end{document}